\documentclass[12pt, twoside]{article}
\usepackage{amsfonts,amsmath,amsthm,amssymb}
\usepackage{enumerate}
\usepackage{xfrac}

\def\titlerunning#1{\gdef\titrun{#1}}
\makeatletter
\def\author#1{\gdef\autrun{\def\and{\unskip, }#1}\gdef\@author{#1}}
\def\address#1{{\def\and{\\\hspace*{18pt}}\renewcommand{\thefootnote}{}%
\footnote {#1}}%
\markboth{\autrun}{\titrun}}
\makeatother
\def\email#1{e-mail: #1}
\def\subjclass#1{{\renewcommand{\thefootnote}{}%
\footnote{\emph{Mathematics Subject Classification (2010):} #1}}}
\def\keywords#1{\par\medskip
\noindent\textbf{Keywords.} #1}

\usepackage{amsfonts,amssymb,amsmath,amsthm}
\usepackage{url}
\usepackage{comment}
\usepackage{subcaption}

\usepackage{tikz}

\usetikzlibrary{arrows}
\usetikzlibrary{petri}
\usetikzlibrary{backgrounds}
\usetikzlibrary{decorations.pathmorphing}
\usetikzlibrary{calc}
\tikzstyle{vert}=[circle,draw=blue!50,fill=blue!20,thick]
\tikzstyle{diedge}=[->,shorten <=1pt,>=angle 90,semithick]
\tikzstyle{myedge}=[->,shorten <=1pt,>=angle 90,semithick]
\tikzstyle{uedge}=[-,shorten <=1pt,>=angle 90]
\tikzstyle{utedge}=[-,shorten <=1pt,>=angle 90,line width=2pt]
\tikzstyle{krouzek}=[rounded corners=4, draw=black, fill=white!20]
\tikzstyle{fkrouzek}=[rounded corners=4, draw=black, fill=black!20]

\newcommand{\MyFig}[1]{\begin{figure}[ht]#1\end{figure}}

\newtheorem{theorem}{Theorem}[section]
\newtheorem*{theorem*}{Theorem}
\newtheorem{lemma}[theorem]{Lemma}
\newtheorem{proposition}[theorem]{Proposition}

\newtheorem{definition}[theorem]{Definition}

\theoremstyle{definition}

\newcommand{\relstr}[1]{\mathbb{#1}}

\newcommand{\alg}[1]{\mathbf{#1}}

\newcommand{\Pol}{\mathop{\mathrm{Pol}}}

\DeclareMathOperator\absorbs{\trianglelefteq}

\DeclareMathOperator\Jabsorbs{\trianglelefteq_J}
\DeclareMathOperator\Neigh{N}
\DeclareMathOperator\Branch{Br}
\newcommand{\NP}{NP}
\newcommand{\AND}{\wedge}

\begin{document}

\titlerunning{Deciding absorption in relational structures}

\title{Deciding absorption in relational structures}

\author{Libor Barto  \and Jakub Bul\'in}



\date{}

\maketitle

\address{Department of Algebra, Faculty of Mathematics and Physics, Charles University in Prague, Sokolovsk\'a 83, 18675 Prague 8, Czech Republic;\email{libor.barto@gmail.com}}
\address{Department of Mathematics, University of Colorado Boulder, Boulder, Colorado 80309-0395, USA; \email{jakub.bulin@gmail.com}}
\thanks{Both authors were supported by the Grant Agency of the Czech Republic, grant 13-01832S.  The second author was also supported by the Polish National Science Centre (NCN) grant 2011/01/B/ST6/01006.
}

\begin{abstract}
We prove that for finite, finitely related algebras the concepts of an absorbing subuniverse and a J\'onsson absorbing subuniverse coincide. Consequently, it is decidable whether a given subset is an absorbing subuniverse of the polymorphism algebra of a given  relational structure.
\end{abstract}

\section{Introduction}

A finite algebra $\alg{A}$  is called \emph{finitely related} if its clone of term operations is equal to the polymorphism clone of some relational structure with finite signature.
In~\cite{Bar13} it was proved that every finite, finitely related algebra with J\'onsson terms has a near unanimity term; consequently, it is decidable whether a given relational structure has a near unanimity polymorphism. 
In the present paper, we generalize this result and its consequence. A partial result in the same direction already appeared in~\cite{Bulin14}.

Absorption (J\'onsson absorption, respectively) generalizes near unanimity operations (J\'onsson operations) in the following way.
A subuniverse $B$ or a subalgebra $\alg{B}$ of an algebra $\alg{A}$ is \emph{absorbing}, written $B \absorbs \alg{A}$, if $\alg{A}$ has an idempotent term $t$ (an \emph{absorption term}) such that
\[
t(B,B,\dots, B, A, B, \dots, B) \subseteq B
\]
for every position of $A$. We also say that $B$ absorbs $\alg{A}$, that $t$ witnesses the absorption $B \absorbs \alg{A}$, etc.

A subuniverse $B$ is called \emph{J\'onsson absorbing}, written $B \Jabsorbs \alg{A}$, if $\alg{A}$ has a sequence of (necessarily idempotent) terms $d_0$, $d_1$, \dots, $d_n$ (a \emph{J\'onsson absorption chain}) such that
\begin{align*}
  &d_i(B,A,B)\subseteq B\text{ for all }i\leq n,\\
  &d_i(x,y,y)\approx d_{i+1}(x,x,y)\text{ for all }i<n,\\
  &d_0(x,y,z)\approx x,\text{ and } d_n(x,y,z)\approx z.
\end{align*}

Note that these terms do not correspond to the standard J\'onsson terms, but to
directed J\'onsson terms introduced in~\cite{KKMM}. However, the definition
via the original J\'onsson terms gives an equivalent concept of J\'onsson absorption~\cite{KKMM}.

The relation of absorption to near unanimity terms and J\'onsson terms is as follows.
Clearly, every singleton subuniverse $\{a\}$ of $\alg{A}$ is absorbing (J\'onsson absorbing, respectively) whenever $\alg{A}$ has a near unanimity term (J\'onsson terms, respectively). A short argument (see eg. \cite{BKazda,Bulin14}) shows that the converse is also true, that is, a finite algebra $\alg{A}$ has a near unanimity term (J\'onsson terms, respectively) if and only if every singleton is an absorbing subuniverse (J\'onsson absorbing subuniverse) of the full idempotent reduct of $\alg{A}$. 

Absorption and J\'onsson absorption played a key role in the proofs of several results concerning
Maltsev conditions and the complexity of the constraint satisfaction problem~(eg.~\cite{BK09a,B11,BK12,BK14,BKS15,BKW12}).
Absorption is stronger than J\'onsson absorption (see Section~\ref{sec:proof}) and is sometimes easier to work with.
Our main theorem shows that, in fact, these concepts coincide for finitely related algebras. A special case, for congruence meet-semidistributive algebras, was proved in \cite{Bulin14}.

\begin{theorem} \label{thm:main}
Let $B$ be a subuniverse of a finite, finitely related algebra $\alg{A}$. Then $B \absorbs \alg{A}$ if and only if $B \Jabsorbs \alg{A}$. 
\end{theorem}

Since it is decidable whether $B$ is a J\'onsson absorbing subuniverse of the polymorphism algebra of a given relational structure $\relstr{A}$, we immediately get the decidability of the corresponding problem for absorption. In fact, if we are guaranteed that the input subset $B$ is a subuniverse, then we can place this problem to the class \NP:

\begin{theorem} \label{thm:complexity}
The problem of deciding whether a given subuniverse of the polymorphism algebra of a given relational structure is (J\'onsson) absorbing, is in \NP.
\end{theorem}

In~\cite{BKazda}, absorption in general finite algebras (not necessarily finitely related) is characterized by means of J\'onsson absorption and cube term blockers. This result is then used to devise an algorithm for the algebraic version of the absorption problem, in which the algebra is given by the tables of the basic operations. However, we are not able to use the results in~\cite{BKazda} to simplify our proof of Theorem~\ref{thm:main}. 

\section{Preliminaries}

We use rather standard universal algebraic terminology~\cite{BS81,Berg11}. Algebras are denoted by capital letters $\alg{A}$, $\alg{B}$, \dots in boldface and relational structures by capitals in blackboard bold $\relstr{A}, \relstr{B}, \dots$. The same letters  $A$, $B$, \dots in the plain font are used to denote their universes. A subuniverse or a subalgebra of a power is called a \emph{subpower}. We mostly work with finite algebras that are idempotent, that is, each  basic operation $f$ satisfies the identity $f(x, x, \dots, x) \approx x$. 

The algebra of polymorphisms of a relational structure $\relstr{A}$ is denoted $\Pol \relstr{A}$. We assume that the reader is familiar with the Pol--Inv Galois correspondence~\cite{G68,BKKR69}. Recall, in particular, that $R$ is a subpower of $\Pol \relstr{A}$
(where $A$ is finite) if and only if $R$ can be pp-defined from $\relstr{A}$, i.e., defined by a formula of the form
\[
(\exists x_1) (\exists x_2)\dots S_1(x_{i_1}, x_{i_2}, \dots) \wedge S_2(x_{j_1}, \dots) \wedge \dots,
\]
where $S_1, S_2, \dots$ are relations in $\relstr{A}$ regarded as predicates in the formula. The clauses in pp-formulas will also be called \emph{constraints} and the tuple of variables in a constraint will be called the \emph{scope} of the constraint. 

The \emph{incidence multigraph} of a $\tau$-structure $\mathbb A$ is defined as a bipartite multigraph with parts $A$ and $\{(R,\bar a)\mid R\in\tau,\bar a\in R\}$ and edge relations $E_1,E_2,\dots$ defined by $(b,(R,\bar a))\in E_i$ if and only if $b=a_i$. 

By the \emph{degree} of $a\in A$ we mean its degree in the incidence multigraph. The \emph{leaves} of $\mathbb A$ are vertices $a\in A$ of degree at most one.  The \emph{neighborhood} of $a$, denoted $\Neigh(a)$, is the set of all vertices $b\in A$, $b\neq a$ which appear in some hyperedge containing $a$ (in other words, it is the set of vertices whose distance from $a$ in the incidence multigraph is two). 

The structure $\mathbb A$ is a \emph{relational tree} if its incidence multigraph is a tree (in particular, it has no multiple edges). For a pair of vertices $a\neq b\in A$ we define the \emph{branch} rooted at $a$ containing $b$, denoted $\Branch(a;b)$, to be the connected component which contains $b$ and is obtained from $\mathbb A$ by removing all but one hyperedges containing $a$.

A pp-formula $\Phi$ in the language $\tau$ can be naturally regarded as a $\tau$-structure (together with a labeling of free variables): its vertices are variables (both free and bound), and the relation corresponding to $R \in \tau$ consists of all tuples $\bar x$ which appear in $\Phi$ as a scope of some constraint of the form $R(\bar x)$. We call $\Phi$ a \emph{tree formula} if the corresponding relational structure is a relational tree, $\Phi$ is \emph{connected} if the structure is connected, etc.

A binary relation $E \subseteq A^2$ is often regarded as a digraph. A sequence $(a_0,a_1, \dots, a_n) \in A^{n+1}$ is called an \emph{$E$-walk} (of length $n$) if $(a_{i-1},a_i) \in E$ for each $i\in[n]$. Such a walk is \emph{closed} if $a_0=a_n$.

\section{Absorption}

In this section we introduce several facts about absorbing and J{\' o}nsson absorbing subuniverses which will be needed later in the proof.

\subsection{Absorption and pp-definitions}

The set of subuniverses of an algebra is closed under pp-definitions. The following simple lemma (see~\cite[Lemma 11]{BKazda}) is a version of this fact for absorption. 

\begin{lemma}\label{lem:abspp}
  Assume that a subpower $R$ of $\alg{A}$ is defined by 
  \[
    R=\{(x_1,\dots,x_n) \mid (\exists y_1\dots\exists y_m)\, R_1(\sigma_1)\AND
    R_2(\sigma_2)\AND\dots\AND R_k(\sigma_k)\},
  \]
  where $R_1,\dots,R_k$ are subpowers of $\alg{A}$
   and $\sigma_1$,\dots,$\sigma_k$ stand for sequences of (free or bound) variables. 
  Let $S_1,\dots,S_k$ be subpowers of $\alg{A}$ such that $S_i\absorbs \alg{R}_i$ (resp.
  $S_i\Jabsorbs \alg{R}_i$) for all $i$. Then the subpower
  \[
    S=\{(x_1,\dots,x_n)\mid (\exists y_1\dots\exists y_m)\, S_1(\sigma_1)\AND
    S_2(\sigma_2)\AND\dots\AND S_k(\sigma_k)\},
  \]
  absorbs (resp. J{\' o}nsson absorbs) $\alg{R}$.
\end{lemma}

\subsection{Essential relations}

An alternative characterization of absorption by means of $B$-essential relations, which is given in this subsection, will be essential in the proof of the main theorem.

When $R \subseteq A^n$ and $i \in \{1, \dots, n\}$, we denote by $\pi_{\widehat i}(R)$ the projection of $R$ onto the coordinates $\{1, \dots, n\} \setminus \{i\}$. 

\begin{definition}
A relation $R \subseteq A^n$ is called \emph{$B$-essential} if $R\cap B^n=\emptyset$ and, for every $i=1,2,\dots n$, 
we have $\pi_{\widehat i} (R)\cap B^{n-1}\neq \emptyset$.
\end{definition}

Observe that if $R \leq \alg{A}^n$ is $B$-essential, then by fixing the first coordinate to $B$, that is, by defining
\[
\{(a_2, \dots, a_n) \mid \exists b \in B,\, (b,a_2, \dots, a_n) \in R\},
\]
we get a $B$-essential subuniverse of $\alg{A}^{n-1}$.  
Therefore, the set of arities of $B$-essential subpowers of an algebra is a downset in $(\mathbb{N},\leq)$.

\begin{proposition}[{\cite[Proposition 16]{BKazda}}] \label{prop:essential}
Let $\alg{A}$ be a finite idempotent algebra and $B \leq \alg{A}$. Then $B \absorbs \alg{A}$ by an $n$-ary term if and only if $\alg{A}$ has no $n$-ary $B$-essential subpower. 
\end{proposition}

Note that using Proposition~\ref{prop:essential} we can define an absorbing subuniverse of $\Pol\relstr{A}$ in a purely relational way: a set $B \subseteq A$ is an absorbing subuniverse of $\Pol\relstr{A}$ if and only if $B$ is pp-definable from $\relstr{A}$ and no $B$-essential relation is pp-definable from $\relstr{A}$. A relational description of J\'onsson absorption can be obtained from Theorem~\ref{thm:jabs_local}.

\subsection{A connectivity lemma}

Absorption is used to absorb connectivity properties of subpowers. A technical lemma of this sort will be used in the final stage of the proof of Theorem~\ref{thm:main}.

\begin{lemma} \label{lem:walk}
Let $\alg{A}$ be a finite algebra, $G$, $H$ subsets of $A$, and $P$, $Q$ subuniverses of $\alg{A}^2$ such that
\begin{itemize}
\item[(i)] $P \Jabsorbs \alg{Q}$,
\item[(ii)] $(a,c) \in Q$ for some $a \in G$ and $c \in H$,
\item[(iii)] for every $c \in G$ there exists $a \in G$ such that $(a,c) \in P$, and
\item[(iv)] for every $a \in H$ there exists $c \in H$ such that $(a,c) \in P$.
\end{itemize}
Then there exists a directed $P$-walk from an element in $G$ to an element in $H$.
\end{lemma}

\begin{proof}
Let $d_0, \dots, d_n$ be a J\'onsson absorption chain witnessing $P \Jabsorbs \alg{Q}$.
Let $a \in G$, $c \in H$ be such that $(a,c) \in Q$. By (iii), there exists a directed $P$-walk to $a$ consisting of elements in $G$ of arbitrary length. Since $A$ is finite, it follows that there exists a $P$-walk to $a$ from an element in a closed $P$-walk and both walks lie entirely in $G$. Similarly, from (iv) it follows that there exists a directed $P$-walk from $c$ to an element in a closed $P$-walk, both walks being in $H$. Choosing appropriate elements in the two closed $P$-walks and large enough $k$ divisible by the lengths of the closed $P$-walks, we get closed $P$-walks
$(e_0, e_1, \dots, e_k=e_0) \in G^{k+1}$, $(f_0, f_1, \dots, f_k=f_0)\in G^{k+1}$ and a $Q$-walk $(e_0=g_0,g_1, \dots, g_k = f_0)\in A^{k+1}$.

Now for each $i\leq n$ we apply the $i$-th term $d_i$ to the columns of the following matrix.
\[
\begin{array}{llll}
e_0 			& e_1 & \dots & e_k = e_0 \\
e_0 = g_0 & g_1 & \dots & g_k = f_0 \\
f_0 			& f_1 & \dots & f_k = f_0
\end{array}
\]
Since consecutive elements in the first and third row are in $P$, consecutive elements in the second row are in $Q$, and $d_i(P,Q,P) \subseteq P$, then the obtained sequence is a $P$-walk from $d_i(e_0,e_0,f_0)$ to $d_i(e_0,f_0,f_0) = d_{i+1}(e_0,e_0,f_0)$.
By concatenating these walks we obtain a $P$-walk from $d_0(e_0,e_0,f_0) = e_0 \in G$ to $d_n(e_0,f_0,f_0) = f_0 \in H$, as required.
\end{proof}

\subsection{A loop lemma}

The Loop Lemma (originally from~\cite{BKN08b}, a simpler proof in~\cite{BK12}) says that a subuniverse $E \leq \alg{C}^2$ intersects the diagonal under some structural assumptions on $E$ and algebraic assumptions on $\alg{C}$. We will need the following variant.

\begin{lemma} \label{lem:loop}
Let $\alg{C}$ be a finite idempotent algebra, $\alg{P}, \alg{Q} \leq \alg{C}^2$, $P \Jabsorbs \alg{Q}$, $\Delta_C = \{(c,c)\mid c \in C\} \subseteq Q$, and assume that there exists a closed $P$-walk. Then $P$ intersects $\Delta_C$.  
\end{lemma}

\begin{proof}
We proceed by induction on $|C|$. The case $|C|=1$ is trivial. Assume that $|C|>1$. It suffices to find a proper subuniverse $B\lneqq \mathbf C$ which contains a closed $P$-walk.

Fix some element $p\in C$ lying in a closed $P$-walk. Let $B_0=\{p\}$ and for $j>0$ define inductively $B_j=\{c\in C\mid (\exists b\in B_{j-1})\,(b,c)\in P\}$. Note that all the sets $B_j$ are nonempty subuniverses of $\mathbf C$. The proof splits into two cases:

First, suppose that there exists $k>0$ such that $B_k=C$. Choose minimal such $k$ and set $B=B_{k-1}$. Since $B=B_{k-1}$ is contained in $B_k$, it follows that for every $b\in B$ there exists $c\in B$ such that $(c,b)\in P$. Consequently, by finiteness of $B$, $B$ must contain a closed $P$-walk. Note that $B\neq C$ follows from the minimality of $k$.

Second, assume that $B_j\neq C$ for all $j\geq 0$. Since there are only finitely many subsets of $C$, there must exist $k<l$ such that $B_k=B_l$. Fix such $k<l$ and define $B=B_k=B_l$. For every $c\in B$ there exists $b\in B$ and a closed $P$-walk of length $(l-k)$ from $b$ to $c$. Since $B$ is finite, it follows that there exist $b_0\in B$ and a closed $P$-walk $(b_0,b_1,\dots,b_{m-1},b_0)\in C^{m+1}$ whose length $m$ is a multiple of $(l-k)$. 

We will show that, in fact, all elements of this walk lie in $B$. Choose $j\in\{1,\dots,m-1\}$. Similarly as in the proof of Lemma \ref{lem:walk}, we apply a J\'onsson absorption chain $d_0, \dots, d_n$ witnessing $P \Jabsorbs \alg{Q}$ to all columns of the following matrix:

\[
\begin{array}{lllllllll}
b_0 & b_1 & \dots & b_{j-1} & b_j & b_{j+1} & \dots & b_{m-1} & b_0\\
b_0 & b_1 & \dots & b_{j-1} & b_j & b_j & \dots & b_j & b_j\\
b_j & b_{j+1} & \dots & \dots & \dots & \dots & \dots & b_{j-1} & b_j
\end{array}
\]
The first and third rows are $P$-walks while the second row is a $Q$-walk (here we use that both $P\subseteq Q$ and $\Delta_C\subseteq Q$). The obtained sequence is a $P$-walk from $d_i(b_0,b_0,b_j)$ to $d_i(b_0,b_j,b_j) = d_{i+1}(b_0,b_0,b_j)$. By concatenating these walks we obtain a $P$-walk from $b_0=d_0(b_0,b_0,b_j)$ to $b_j=d_n(b_0,b_j,b_j)$ whose length is a multiple of $(l-k)$. As $b_0\in B=B_k$, this shows that $b_j\in B_l=B$ which concludes the proof.
\end{proof}

\section{Proof of the main result} \label{sec:proof}

In this section we prove the main result, Theorem~\ref{thm:main}.

Let $\relstr{A}$ be a finite relational structure with finite signature, $\alg{A} = \Pol \relstr{A}$, and let $B$ be a subuniverse of $\alg{A}$. 
We aim to show that $B \absorbs \alg{A}$ if and only if $B \Jabsorbs \alg{A}$. 

The left-to-right implication is simple: if $t$ is an $n$-ary absorption term witnessing $B \absorbs \alg{A}$ then the sequence
$d_0, \dots, d_n$ defined by
\[
d_i(x,y,z) = t(\underbrace{z, \dots, z}_{i \times}, y, x, \dots, x)
\]
is a J\'onsson absorption chain witnessing $B \Jabsorbs \alg{A}$. 

Assume now that $B \Jabsorbs \alg{A}$. Since J\'onsson absorption terms are idempotent, we can assume that $\relstr{A}$ contains all the singleton unary relations and $\alg{A}$ is thus an idempotent algebra.
 We aim to show that $B \absorbs \alg{A}$. In fact, we will prove the following refinement.

\begin{theorem} \label{thm:core}
Let $\relstr{A}$ be a finite relational structure which contains all the singleton unary relations and whose relations all have arity at most $\theta$ (where $\theta\geq 2$). Assume that $B$ is a J\'onsson absorbing subuniverse of $\alg{A} = \Pol \relstr{A}$. Then $B \absorbs \alg{A}$ via a term of arity at most $\kappa=\frac{1}{2}(2\theta-2)^{3^{|A|}}+1$.
\end{theorem}

Assume the contrary. By Proposition~\ref{prop:essential}, $\alg{A}$ has a $B$-essential subpower $U$ of arity $\kappa$. Since $U$ is a subpower of $\alg{A}$, it can be pp-defined from $\relstr{A}$. The overall structure of the proof is as follows.

\begin{itemize}
\item We introduce some simplifying assumptions on $\mathbb A$ and the structure of the pp-formula defining $U$ (Subsection \ref{subsec:preproc}). These assumptions will be applied throughout our proof.
\item Then we use Zhuk's technique from~\cite{Zhuk14}. By a repeated application of the construction described in Subsection \ref{subsec:surgery} we modify the formula  to obtain a tree formula defining a (possibly different) $B$-essential subpower of the same arity $\kappa$.
\item In Subsection \ref{subsec:comb} we use the tree formula to obtain a structurally very simple pp-definition (a ``comb'' formula) of a $B$-essential subpower of a smaller, but still large arity $\lambda$ (where $\lambda$ corresponds to the length of a maximal path in the tree formula; roughly a logarithm of $\kappa$).
\item This simple pp-definition is then used to obtain a pair of subalgebras $\alg{P} \Jabsorbs \alg{Q}$ of $\alg{A}^2$ that will contradict Lemma~\ref{lem:loop} (Subsection \ref{subsec:contradiction}).

\end{itemize}

\subsection{Preprocessing}\label{subsec:preproc}
During the construction, some relations in pp-formulas are replaced by subpowers of $\alg{A}$ (of arity at most $\theta$) which are not necessarily in $\relstr{A}$. For this technical reason, we assume that each relation of arity at most $\theta$ which is pp-definable from $\relstr{A}$ is already in $\relstr{A}$. Note that we can safely make this assumption since adding pp-definable relations does not change the algebra of polymorphisms $\alg{A}$.

Using the above assumption we can modify the formula defining $U$ so that it satisfies the following properties. 
The resulting formula will be called a \emph{simplified form} of the original formula.
\begin{enumerate}[(1)]
\item The formula is connected.
\item The set of free variables is equal to the set of leaves.
\item All bound variables have degree 2 or 3.
\item All scopes of constraints are non-repeating sequences of variables.
\item There are no unary constraints.
\end{enumerate}
For item (1) note that if there were two free variables in different connected components, then the resulting relation $U$ would be a product of two relations of smaller arities. This cannot be the case for a $B$-essential relation. Components containing only bound variables are trivial and can be omitted. Items (2)-(5) can be achieved by a combination of introducing new variables and replacing relations by other, pp-definable relations. See \cite{Bvaleriote} for a detailed description of a similar simplification process in the case $\theta=2$. 
 
Let us remark here that with some weakening of the upper bound on $\kappa$ one could make a further simplifying assumption: namely that $\theta=2$ (this is done in \cite{Bar13,Bvaleriote,Bulin14} as well as in our illustrative figures) and even that $\relstr{A}$ contains a single binary relation (see \cite[Section 7]{Bulin14} for details). However, here we choose to present the proof without these reductions.

\subsection{Zhuk's surgery}\label{subsec:surgery}

Recall that $U \leq \alg{A}^{\kappa}$ is a $B$-essential relation. 
We take a simplified form of a pp-formula defining $U$ from $\relstr{A}$. We denote by $\Phi$ the quantifier-free part of this formula, and denote $x_1, \dots, x_{\kappa}$ the free variables. For an arbitrary sequence $y_1, \dots, y_m$ of variables in $\Phi$, we denote $\Phi(y_1,\dots,y_m)$ both the formula obtained by existentially quantifying over the remaining variables, and the $m$-ary relation defined by this formula. In particular, we have $U = \Phi(x_1, \dots, x_\kappa)$.

We describe a construction which transforms $\Phi$ into a new quantifier-free pp-formula $\Phi'$ such that $\Phi'(x_1, \dots, x_{\kappa})$ is still a $B$-essential relation of arity $\kappa$, possibly different from $U = \Phi(x_1, \dots, x_{\kappa})$. 

The construction depends on a variable $y \not\in \{x_1, \dots, x_{\kappa}\}$ and a constraint in $\Phi$ whose scope contains $y$, say $T(\dots, y, \dots)$. 
 
The construction is divided into three steps.
In the first step, we build from $\Phi$ a new formula $\Psi$ by 
adding a new variable $y_*$, removing the 
constraint $T(\dots, y, \dots)$, adding the constraint $T(\dots, y_*, \dots)$, and adding the unary constraints $C(y)$ and $C(y_*)$, where $C = \Phi(y)$. Figure~\ref{fig:first_step} depicts this modification in the case that our constraint is binary, of the form $T(y,z)$.

\MyFig{
  \centering
  \begin{subfigure}{.5\textwidth}
  \centering
  \begin{tikzpicture}
     \filldraw[fill=gray!20,draw=black] (1.6,4) -- (1.6,1) -- (4.4,1) -- (4.4,2.5) .. controls (2.8, 2.4) .. (1.6,4);
		 \draw (4.4,2.5) node[krouzek] {$z$} -- (1.6,4) node[krouzek] {$y$} node[midway, above=0pt] {$T$};
		 \draw (2,1) -- (2,0) node[krouzek] {$x_1$};
		 \draw (3,1) -- (3,0) node[krouzek] {$x_2$};
		 \draw (4,1) -- (4,0) node[krouzek] {$x_3$};
  \end{tikzpicture}
	\caption{Original formula $\Phi$.}
	\end{subfigure}%
	\begin{subfigure}{.5\textwidth}
	\centering
  \begin{tikzpicture}
     \filldraw[fill=gray!20,draw=black] (1.6,4) -- (1.6,1) -- (4.4,1) -- (4.4,2.5) .. controls (2.8, 2.4) .. (1.6,4) node[krouzek] {$y$} node[right=6pt] {$C$};
		 \draw (4.4,2.5) node[krouzek] {$z$} -- (5.5,4) node[krouzek] {$y_{*}$} node[midway, above left=-3pt] {$T$} node[right=8pt] {$C$};
		 \draw (2,1) -- (2,0) node[krouzek] {$x_1$} node[above left=3pt] {}; 
		 \draw (3,1) -- (3,0) node[krouzek] {$x_2$} node[above left=3pt] {}; 
		 \draw (4,1) -- (4,0) node[krouzek] {$x_3$} node[above left=3pt] {}; 
  \end{tikzpicture}
	\caption{Resulting formula $\Psi$.}
	\end{subfigure}
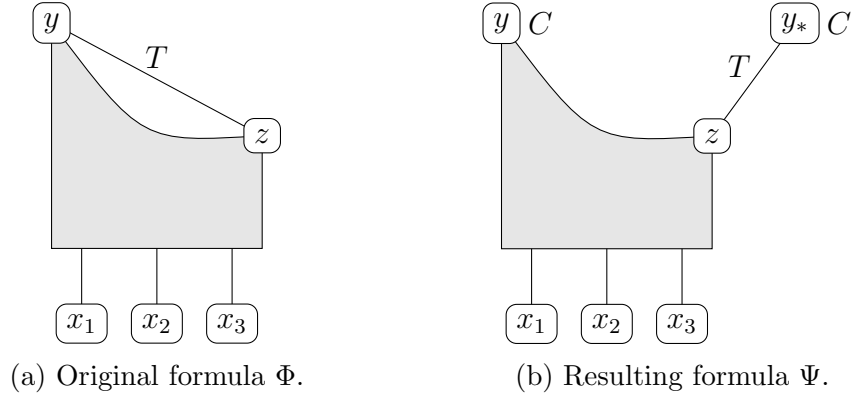
  \caption{The first step for $\kappa=3$.}
  \label{fig:first_step}
}

In the second step, we define a new formula $\Theta$ as follows. For each $i \in [l]$, where $l = |A|$, we take a copy $\Psi^i$ of the formula $\Psi$ by renaming each variable $w$ in $\Psi$ to $w^i$. Then we take the conjunction of $\Psi^1$, \dots, $\Psi^l$ and identify variables $y^i_*$ and $y^{i+1}$ for each $i<l$.
The resulting formula is shown in Figure~\ref{fig:second_step}.

\MyFig{   
  \begin{center}
    \begin{tikzpicture}
     \filldraw[fill=gray!20,draw=black] (1.6,4) -- (1.6,1) -- (4.4,1) -- (4.4,2.5) .. controls (2.8, 2.4) .. (1.6,4) node[krouzek] {$y^1$} node[right=8pt] {$C$};
		 \draw (4.4,2.5) node[krouzek] {$z^1$} -- (5.5,4) node[midway, left=1pt] {$T$};
		 \draw (2,1) -- (2,0) node[krouzek] {$x_1^1$} node[above left=4pt] {}; 
		 \draw (3,1) -- (3,0) node[krouzek] {$x_2^1$} node[above left=4pt] {}; 
		 \draw (4,1) -- (4,0) node[krouzek] {$x_3^1$} node[above left=4pt] {}; 
     \filldraw[fill=gray!20,draw=black] (1.6+3.9,4)  -- (1.6+3.9,1) -- (4.4+3.9,1) -- (4.4+3.9,2.5) .. controls (2.8+3.9, 2.4) .. (1.6+3.9,4) node[krouzek] {$y^1_*=y^2$} node[right=22pt] {$C$};
		 \draw (4.4+3.9,2.5) node[krouzek] {$z^2$} -- (5.5+3.9,4)  node[midway, left=1pt] {$T$};
		 \draw (2+3.9,1) -- (2+3.9,0) node[krouzek] {$x_1^2$} node[above left=4pt] {}; 
		 \draw (3+3.9,1) -- (3+3.9,0) node[krouzek] {$x_2^2$} node[above left=4pt] {}; 
		 \draw (4+3.9,1) -- (4+3.9,0) node[krouzek] {$x_3^2$} node[above left=4pt] {}; 
		\filldraw[fill=gray!20,draw=black] (1.6+7.8,4)  -- (1.6+7.8,1) -- (4.4+7.8,1) -- (4.4+7.8,2.5) .. controls (2.8+7.8, 2.4) .. (1.6+7.8,4) node[krouzek] {$y^2_*=y^3$} node[right=22pt] {$C$};
		 \draw (4.4+7.8,2.5) node[krouzek] {$z^3$} -- (5.5+7.8,4) node[krouzek] {$y_{*}^3$} node[midway, left=1pt] {$T$} node[left=7pt] {$C$};
		 \draw (2+7.8,1) -- (2+7.8,0) node[krouzek] {$x_1^3$} node[above left=4pt] {}; 
		 \draw (3+7.8,1) -- (3+7.8,0) node[krouzek] {$x_2^3$} node[above left=4pt] {}; 
		 \draw (4+7.8,1) -- (4+7.8,0) node[krouzek] {$x_3^3$} node[above left=4pt] {}; 
  \end{tikzpicture}
  \end{center}
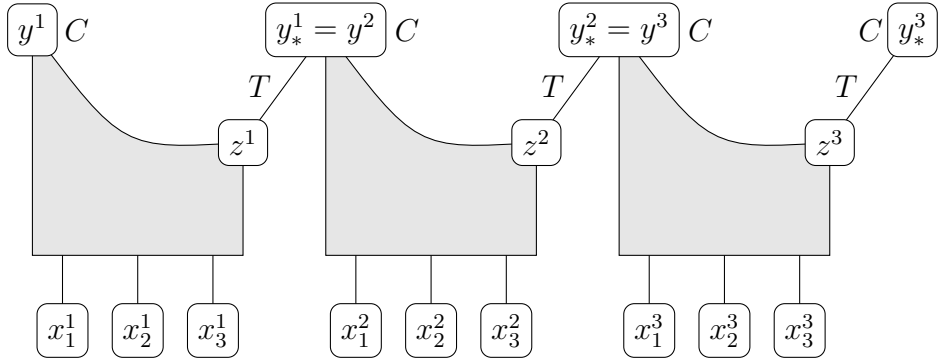
  \caption{The second step: formula $\Theta$ for $l=\kappa=3$}
  \label{fig:second_step}
}

Before describing the third step, 
we show that the relation
$$
V = \Theta(x^1_1, x^2_1, \dots, x^l_1,x^1_2, \dots, x^l_2,\dots,x^1_\kappa,\dots, x^l_{\kappa})
$$
is ``block-wise $B$-essential'':

\begin{lemma} \label{lem:step_two}
The relation $V$ does not intersect $B^{l\kappa}$ but, for every $i \in [\kappa]$, $V \cap (B^{l(i-1)} \times A^{l} \times B^{l(\kappa-i)}) \neq \emptyset$.
\end{lemma}


\begin{proof}
By construction, the tuple 
\[
(\underbrace{a_1, \dots, a_1}_{l \times}, \underbrace{a_2, \dots, a_2}_{l \times}, \dots, \underbrace{a_{\kappa}, \dots a_{\kappa}}_{l \times})\]
is in $V$ whenever $(a_1, \dots, a_{\kappa})$ is in $U$. Therefore, the second claim is true and it is enough to check that $V \cap B^{l\kappa}$ is empty. Suppose otherwise, that is, there exists $(b_1, \dots, b_{l\kappa}) \in V \cap B^{l\kappa}$. 

Let $S = \Psi(y,y_*,x_1, \dots, x_{\kappa})$ and let $Q$ and $P$ be the projections of $S$ and $S \cap (C^2 \times B^{\kappa})$ onto the first two coordinates, respectively. By Lemma~\ref{lem:abspp}, $P \Jabsorbs \alg{Q}$. By construction of $\Psi$, we have that $\Delta_C \subseteq Q$ and that the projection of $S\cap (\Delta_C \times A^{\kappa})$ onto the coordinates $3,4, \dots, \kappa+2$ is equal to $U=\Phi(x_1, \dots, x_\kappa)$. By construction of $\Theta$ from $\Psi$, the tuple $(b_1, \dots, b_{l\kappa}) \in V \cap B^{l\kappa}$ implies the existence of a $P$-walk of length $l$. Since $l=|A|$ this walk must contain a closed $P$-walk.

The assumptions of Lemma~\ref{lem:loop} are satisfied, and therefore $P$ contains a loop $(c,c)$. But this gives us a tuple $(c,c,b_1,\dots,b_\kappa)\in S\cap(\Delta_C\times B^\kappa)$ which means that $(b_1,\dots,b_\kappa)\in U\cap B^\kappa$. This contradicts the $B$-essentiality of $U$.
\end{proof}

In the third step, we modify $\Theta$ using the data provided by the following lemma. 
\begin{lemma} \label{lem:step_three}
Let $l_1, \dots, l_{\kappa}$ be positive integers, $s=l_1+\dots+l_\kappa$, and let $W \leq \alg{A}^s$ be such that $W$ does not intersect $B^s$ but, for each $i \in [\kappa]$, intersects $B^{l_1+\dots+l_{i-1}} \times A^{l_i} \times B^{l_{i+1} + \dots + l_\kappa}$.

Then there exist $m_i \in [l_i]$ for each $i \in [\kappa]$ and $C_{i}^j \leq \alg{A}$ for each $i \in [\kappa], j \in ([l_i] \setminus \{m_i\})$ such that
\begin{align*}
U' = 
\{(a_1^{m_1}, \dots, a_\kappa^{m_\kappa})\mid &(a^1_1, \dots, a^{l_1}_1, a^1_2, \dots, a^{l_2}_2, \dots, a^1_\kappa, \dots, a^{l_\kappa}_\kappa) \in V,  \\
&(\forall i \in [\kappa]) (\forall j \in [l_i], j \neq m_i) \  a_i^j \in C_i^j\}
\end{align*}
is a $B$-essential relation.
\end{lemma}

\begin{proof}
We prove the claim by induction on $s$. The base case $l_1=l_2=\dots=l_{\kappa}=1$ is trivially true as $U'=W$ is already $B$-essential. 

For the induction step assume without loss of generality that $l_\kappa>1$. There are two cases:
\begin{itemize}
\item If $W\cap (B^{s-1}\times A)=\emptyset$, then we can apply the induction hypothesis to $l_1,\dots,l_{\kappa-1},l_\kappa-1$ and the projection of $W$ onto the first $s-1$ coordinates. Let $m_i$, $C^i_j$ 
 be the data provided and, in addition, set $C_\kappa^\kappa=A$.
\item If $W\cap (B^{s-1}\times A)\neq\emptyset$, then we set $m_\kappa=l_\kappa$, $C^\kappa_1=\dots=C^\kappa_{l_\kappa-1}=B$. For $i<\kappa$ we use the data from applying the claim to $l_1,\dots,l_{\kappa-1},1$ and the projection of $W$ onto the first $l_1+\dots+l_{\kappa-1}$ coordinates and the last coordinate.
\end{itemize}
It is easy to verify that in both cases the resulting relation $U'$ is indeed $B$-essential.
\end{proof}

We apply this lemma for $W=V$, $l_1 = \dots = l_{\kappa}=l$. Note that the assumptions are satisfied by Lemma~\ref{lem:step_two}.

The formula $\Phi'$ is obtained from $\Theta$ as follows. For each $i \in [\kappa]$, we rename the variable $x_i^{m_i}$ to $x_i$ and, for each $i \in [\kappa]$, $j \in [l]$, $j \neq m_i$, we add
the constraint $C_i^j(x_i^j)$. It follows from Lemma~\ref{lem:step_three} that $U' =  \Phi'(x_1, \dots, x_\kappa)$ is a $B$-essential relation.

Zhuk's proof~\cite{Zhuk14} (see also~\cite{Bvaleriote} for an alternative presentation in a slightly simplified situation that $\theta=2$) now shows that by a repeated application of this construction with suitable choices of variables and constraints (plus the simplifications from Subsection \ref{subsec:preproc}) we can obtain a tree formula $\Phi$ such that 
$\Phi(x_1, \dots, x_\kappa)$ is simplified and defines a $B$-essential relation. 

\subsection{A comb formula}\label{subsec:comb}

Recall that by fixing some coordinates of a $B$-essential relation to $B$ we get a $B$-essential relation. It follows that if we select $i_1, \dots, i_\lambda \in [\kappa]$ arbitrarily and form $\Phi'$ by adding to $\Phi$ the unary constraints $B(x_{j})$, $j \in [\kappa] \setminus \{i_1, \dots, i_{\lambda}\}$, then $\Phi'(x_{i_1}, \dots, x_{i_{\lambda}})$ will be $B$-essential. 

The following lemma shows that these variables can be selected so that $\lambda$ is roughly a logarithm of $\kappa$ and $\Phi'$ can be simplified to an especially nice form (see also the proof of~\cite[Theorem~6.1]{Zhuk14}, or section ``Comb definition'' of~\cite{Bvaleriote} for a proof in the case $\theta=2$).

\begin{lemma} \label{lem:comb}
	There exists a pp-formula over $\relstr{A}$ of the form
	\[
	(\exists w_1 \,\dots\,\exists w_{\lambda+1}) \,S_1(z_1,w_1,w_2) \wedge S_2(z_2,w_2, w_3) \wedge \dots \wedge S_{\lambda}(z_{\lambda}, w_{\lambda}, w_{\lambda+1})
	\]
	that defines a $B$-essential relation $R$ of arity $\lambda \geq 3^{|A|}$  (see Figure~\ref{fig:comb}).
\end{lemma}

\MyFig{
  \begin{center}
    \begin{tikzpicture}
	     \draw[rounded corners=5mm,fill=gray!20] (-0.5,0)--(-0.5,2.5)--(2.5,2.5)--(2.5,1.5)--(0.5,1.5)--(0.5,0)--cycle node[above = 80pt,right = 30pt] {$S_1$};
	     \draw[rounded corners=5mm,fill=gray!10] (1.5,0)--(1.5,2.5)--(4.5,2.5)--(4.5,1.5)--(2.5,1.5)--(2.5,0)--cycle node[above = 80pt,right = 30pt] {$S_2$};
	     \draw[rounded corners=5mm,fill=gray!20] (3.5,0)--(3.5,2.5)--(6.5,2.5)--(6.5,1.5)--(4.5,1.5)--(4.5,0)--cycle node[above = 80pt,right = 30pt] {$S_3$};
	     \draw[rounded corners=5mm,fill=gray!10] (5.5,0)--(5.5,2.5)--(8.5,2.5)--(8.5,1.5)--(6.5,1.5)--(6.5,0)--cycle node[above = 80pt,right = 30pt] {$S_4$};
	     \draw[rounded corners=5mm,fill=gray!20] (7.5,0)--(7.5,2.5)--(10.5,2.5)--(10.5,1.5)--(8.5,1.5)--(8.5,0)--cycle node[above = 80pt,right = 30pt] {$S_5$};
		 \draw(0,2)  (0,0.5) node[krouzek] {$z_1$};
		 \draw(2,2)  (2,0.5) node[krouzek] {$z_2$};
		 \draw(4,2)  (4,0.5) node[krouzek] {$z_3$};
		 \draw(6,2)  (6,0.5) node[krouzek] {$z_4$};
		 \draw(8,2)  (8,0.5) node[krouzek] {$z_5$};
		 \draw(0,2) node[krouzek] {$w_1$}  (2,2) node[krouzek] {$w_2$}  (4,2) node[krouzek] {$w_3$}  (6,2) node[krouzek] {$w_4$}  (8,2) node[krouzek] {$w_5$}  (10,2) node[krouzek] {$w_6$};
		 
  \end{tikzpicture}
  \end{center}
  \caption{A comb definition of the relation $R$.}
  \label{fig:comb}
}
We remark that a similar formula from~\cite[Theorem~6.1]{Zhuk14} has two additional binary constraints. The difference is negligible.

\begin{proof}
Let $\Phi(x_1, \dots, x_{\kappa})$ be a simplified tree formula defining a $B$-essential relation.
We start by fixing an arbitrary leaf $z_1\in\{x_1,\dots,x_\kappa\}$. For every variable $u\neq z_1$ we define $L(u)$ to be the number of leaves (that is, free variables) lying outside of $\Branch(u; z_1)$.  

In the first step of our construction we introduce a new variable $w_1$, add the constraint $w_1=z_1$, and choose $w_2\in\Neigh(z_1)$ so that $L(w_2)$ is maximal. Let $i>1$ and assume that $z_{i-1}$, $w_{i-1}$ and $w_i$ are already defined. The construction splits into several cases:

\begin{enumerate}[(1)]
\item If $w_i$ is a leaf, then we set $\lambda = i-1$ and stop the construction process.
\item Else, if $\Neigh(w_i)\setminus\Branch(w_i;w_{i-1})$ contains a single element $v$, then we redefine $w_i=v$ and repeat this construction step.

\item Else, we choose $w_{i+1}\in\Neigh(w_i)\setminus\Branch(w_i;w_{i-1})$ so that $L(w_{i+1})$ is maximal, pick any leaf $z_i\in\Branch(w_{i+1};w_i)\setminus\Branch(w_i;w_{i-1})$, $z_i\neq w_{i+1}$, and increment $i$.
\end{enumerate}
Let $\Phi'$ be the formula obtained from $\Phi\wedge(w_1=z_1)$ by fixing to $B$ all the free variables of $\Phi$ that have not been selected, that is, by adding the unary constraints $B(x_j)$ whenever $x_j\notin\{z_1,\dots,z_\lambda\}$.

The relations $S_i(z_i,w_i,w_{i+1})$ for $1<i<\lambda$ are obtained by taking the induced substructure of $\Phi'$ on the set of variables $(\Branch(w_{i+1}; w_i)\setminus\Branch(w_i;w_{i-1}))\cup\{w_i\}$ and existentially quantifying over all variables except for $z_i$, $w_i$ and $w_{i+1}$. The relation $S_1$ is obtained similarly using the set of variables $\Branch(w_2;w_1)$.

It is easy to see that the resulting ``comb'' formula pp-defines the same relation as $\Phi'$ which is $B$-essential by the argument from the beginning of this subsection. 

It remains to verify that $\lambda\geq 3^{|A|}$. In the first step, as $z_1$ is a leaf, we have that $|\Neigh(z_1)|\leq \theta-1$ and  $\sum_{u\in\Neigh(z_1)}L(u)=\kappa-1$. Note that redefining of $w_i$ in case (2) does not change the value of $L(w_i)$. It follows that $L(w_2)\geq\frac{\kappa-1}{\theta-1}$. In the consequent steps, the degree of $w_i$ is at most 3. Thus there are at most $2(\theta-1)$ neighbors to choose from and $L(w_{i+1})\geq\frac{L(w_i)}{2(\theta-1)}$. This gives us a lower bound 
$$
L(w_\lambda)\geq\frac{\kappa-1}{2^{\lambda-2}(\theta-1)^{\lambda-1}}.
$$
The fact that the construction stopped at $\lambda$ means that $L(w_\lambda)\leq 2(\theta-1)$. Comparing these bounds gives us
$$
\kappa\leq \frac{1}{2}(2\theta-2)^\lambda+1
$$
which implies that $\lambda\geq 3^{|A|}$. 
\end{proof}

\subsection{Contradiction} \label{subsec:contradiction}
Take a pp-definition of a $B$-essential relation $R \leq \alg{A}^{\lambda}$ as in Lemma~\ref{lem:comb}.

The following terminology will be useful. An \emph{$(i,j)$-path from $a$ to $a'$ supported by $(c_{i}, \dots, c_{j-1})$}, where $1\leq i<j\leq \lambda+1$ and $a,a',c_{i}, \dots, c_{j-1} \in A$, is a tuple $(a = a_i,a_{i+1}, \dots, a_j = a')$ of elements of $A$ such that $(c_{l},a_{l},a_{l+1}) \in S_{l}$ for every $l \in \{i, \dots, j-1\}$. Let $\Pi_{i,j}(u,v,\bar z)$ be the pp-formula describing the existence of an $(i,j)$-path from $u$ to $v$ supported by $\bar z$. Observe that a tuple $(c_1, \dots, c_\lambda)$ is in $R$ if and only if there exists some $(1,\lambda+1)$-path supported by $(c_1, \dots, c_\lambda)$, that is, 
$$
\bar c\in R\ \text{ if and only if }\ (\exists u\,\exists v)\,\Pi_{1,\lambda+1}(u,v,\bar c).
$$

For each $i \in \{2,\dots,\lambda\}$ we define two subuniverses $G_i, H_i$ of $\alg{A}$: $a \in G_i$ if and only if there exists a $(1,i)$-path to $a$ supported by a tuple from $B^{i-1}$, and $a \in H_i$ if and only if there exists an $(i,\lambda+1)$-path from $a$ supported by a tuple from $B^{\lambda-i+1}$. The sets $G_i$ and $H_i$ are indeed subuniverses of $\alg{A}$, their pp-definitions are the following:
\begin{align*}
&a\in G_i\ \text{ if and only if }\ (\exists u\,\exists\bar z)\,\Pi_{1,i}(u,a,\bar z)\wedge B^{i-1}(\bar z)\\
&a\in H_i\ \text{ if and only if }\ (\exists v\,\exists\bar z)\,\Pi_{i,\lambda+1}(a,v,\bar z)\wedge B^{\lambda-i+1}(\bar z).
\end{align*}
For every $i \in \{2,\dots,\lambda\}$, the subuniverses $G_i$ and $H_i$ are nonempty (since $B^{i-1} \times A^{\lambda-i+1} \cap R \neq \emptyset$ and $A^{i-1} \times B^{\lambda-i+1} \cap R \neq \emptyset$) and disjoint (since $B^\lambda \cap R = \emptyset$).

Now we use the fact that the arity $\lambda$ of $R$ is large. There are $(3^{|A|}-2^{|A|+1}+1)$ ordered pairs of disjoint nonempty subsets of $A$ and $\lambda-1$ choices for $i$, so, since $\lambda-1 \geq 3^{|A|}-1 > (3^{|A|}-2^{|A|+1}+1)$, there must be two different $k < l$ such that $(G_k,H_k) =(G_l,H_l)$.  We fix such $k,l$ and denote 
\[
G = G_k = G_l, \quad H = H_k = H_l\enspace.
\] 
Finally, we define a subuniverse $Q\leq\mathbf A^2$ so that $(a,c) \in Q$ if and only if there exists a $(k,l)$-path from $a$ to $c$,
and a subuniverse $P\leq\alg{A}^2$ if and only if there exists a $(k,l)$-path from $a$ to $c$ supported by a tuple from $B^{l-k}$. 

We verify the assumptions of Lemma~\ref{lem:walk}.
\begin{itemize}
\item $P$ and $Q$ are subuniverses of $\alg{A}^2$ and $P \Jabsorbs Q$: This follows from Lemma~\ref{lem:abspp} since $B \Jabsorbs \alg{A}$ and these sets can be defined by the pp-formulas
\begin{align*}
&(u,v)\in Q\ \text{ if and only if }\ (\exists\bar z)\,\Pi_{k,l}(u,v,\bar z)\\
&(u,v)\in P\ \text{ if and only if }\ (\exists\bar z)\,\Pi_{k,l}(u,v,\bar z)\wedge B^{k-l}(\bar z)\\
\end{align*}
\item 
$(a,c) \in Q$ for some $a \in G$ and $c \in H$: As $R$ is $B$-essential, there exists a $(1,\lambda+1)$-path $(a_1, \dots, a_{\lambda+1})$ supported by a tuple from $B^{k-1}\times A^{l-k} \times B^{\lambda-l+1}$.
Then $a_k \in G_k=G$, $a_l \in H_l=H$, and $(a_k,a_l) \in Q$ by definitions, so we can set $a=a_k$ and $c=a_l$ to satisfy the claim.
\item 
For every $c \in G$ there exists $a \in G$ such that $(a,c) \in P$: Consider a vertex $c \in G = G_l$. By the definition of $G_l$, there exists a $(1,l)$-path $(a_1,\dots, a_{l-1},a_l=c)$ supported by a tuple from $B^{l-1}$, thus $a_k \in G_k = G$ and $(a_k,c) \in P$. This shows that $a=a_k$ satisfies the claim. 
\item 
For every $a \in H$ there exists $c \in H$ such that $(a,c) \in P$: The proof is similar to the previous item.
\end{itemize}
By Lemma~\ref{lem:walk} there exists a directed $P$-walk from $G$ to $H$. Since $G$ and $H$ are disjoint, the following item contradicts the conclusion of Lemma~\ref{lem:walk} and finishes the proof of Theorem~\ref{thm:core} and Theorem~\ref{thm:main}.

\begin{itemize}
\item 
If $a \in G$ and $(a,c) \in P$, then $c \in G$:
If $a \in G = G_k$ and $(a,c) \in P$, then there exists a $(1,k)$-path $(a_1, \dots, a_k=a)$ supported by $(b_1, \dots, b_{k-1}) \in B^{k-1}$ and a $(k,l)$-path $(a=a_k,a_{k+1}, \dots, a_l=c)$ supported by $(b_k, \dots, b_{l-1})$. Then $(a_1, \dots, a_l)$ is a $(1,l)$-path supported by $(b_1,\dots, b_{l-1})$, hence $a_l=c \in G_l = G$. This proves the claim. 
\end{itemize}

\subsection{Remarks on the proof} \label{subsec:barto_vs_zhuk}

Zhuk's approach used in this paper is different from the proof of the special (congruence meet-semidistributive) case~\cite{Bulin14} generalizing the proof in~\cite{Bar13} that a finite, finitely related algebra with J\'onsson terms has a near unanimity term. In~\cite{Bar13, Bulin14}, a tree definition of a suitable relation was easily derived from existing results on the constraint satisfaction problem~\cite{BK09a,BK14}.

Using Marcin Kozik's improvement of the algebraic machinery from~\cite{BKW12}, it is possible to prove the main result in a cleaner way, similar to~\cite{Bar13}. We have chosen the presented approach since the required results of Kozik are not yet written up and our proof is also a bit more elementary. 

Our upper bound on $\kappa$ slightly improves upon the related results in \cite{Zhuk14,Bar13}. In Zhuk's original paper \cite{Zhuk14} the upper bound on the arity of a near-unanimity operation is $((|A|-1)(\theta-1))^{3^{|A|}}+2$. In \cite{Bar13,Bvaleriote} and the present paper the factor $(|A|-1)$ is reduced to 2 by a simple trick of limiting the maximum degree of the tree formula to 3 (see Subsection \ref{subsec:preproc}). In \cite{Bar13,Bulin14} the upper bound (for $\theta=2$) is $4^{8^{|A|}}+1$. The factor $8^{|A|}$ comes from a pigeonhole argument counting all triples of subsets of $A$. This is suboptimal (as already discussed in \cite[Lemma 6.5]{Bar13}) and can in fact be readily improved to $3^{|A|}$.

On the other hand, Zhuk's paper \cite{Zhuk14} provides the following lower bounds (which can be used in our case too): $\kappa\geq(\theta-1)^{2^{|A|-2}}$ for $\theta\geq 3$ and $|A|\geq 3$, and $\kappa\geq2^{2^{|A|-3}}$ for $\theta=2$ and $|A|\geq 4$.

\section{The complexity of deciding absorption}

We now discuss the computational complexity of the absorption problem whose input is a relational structure $\relstr{A}$ (relations are given by lists of tuples) and a subuniverse $B$ of $\alg{A} = \Pol \relstr{A}$; the problem is to decide whether $B \absorbs \alg{A}$. (Recall that we can without loss of generality assume that $\relstr{A}$ contains all the singletons. Also, the assumption that $B$ is a subuniverse is not limiting for the decidability question: a naive algorithm puts testing whether $B$ is a subuniverse in co-NEXPTIME \cite{Willard10}.)

An ineffective approach, still good enough to show that the absorption problem is decidable, is based on Theorem~\ref{thm:core} which implies that it is enough to check operations of arity $\kappa = \frac{1}{2}(2\theta-2)^{3^{|A|}}+1$.

A better strategy is to directly use the main theorem, Theorem~\ref{thm:main}, which shows that it is enough to verify $B \Jabsorbs \alg{A}$. A straightforward approach would be to find all ternary idempotent polymorphisms of $\relstr{A}$ and look for a J\'onsson absorbing chain among them. A better way to test whether $B \Jabsorbs \alg{A}$ is to use the following local characterization of J\'onsson absorption from \cite{BKazda}.

\begin{theorem}\label{thm:jabs_local}
  Let $\alg{A}$ be a finite idempotent algebra and $B\leq \alg{A}$. Then $B\Jabsorbs
  \alg{A}$ if and only if 
  for every $a,c,d\in A$ and every $b_1,b_2\in B$, the digraph $\relstr{G}=(A,E)$ with the
  edge set
  \[
    E=\{(u,v)\mid \exists b\in B, (b,u,v)\in \langle(b_1,a,a), (b_2,c,c),
    (d,a,c)\rangle\}
  \]
contains a directed path from $a$ to $c$.
\end{theorem} 

The subuniverse $R=\langle(b_1,a,a)$, $(b_2,c,c)$, $(d,a,c)\rangle \leq \alg{A}^3$ 
can be visualised as a colored digraph: an element $(b,u,v)$ is regarded as the edge $(u,v)$ colored by $b$. 
The theorem asks for a $B$-colored directed path from $a$ to $c$.

Theorem~\ref{thm:jabs_local} gives us a straightforward NP algorithm to test whether a given subuniverse $B$ is J{\' o}nsson absorbing as claimed in Theorem~\ref{thm:complexity}: We try all possible $a,c,d\in A$,
$b_1,b_2\in B$ (there are polynomially many such choices).  For each choice we need a polynomial witness for existence of a $B$-colored path from $a$ to $c$. Note that such a path, if it exists, can be given by polynomially many (at most $|A|$) triples from $R$. The fact that a certain triple $(b,u,v)$ lies in $R$ can be witnessed by giving a table of some ternary polymorphism $\varphi$ that generates $(b,u,v)$ from the tuples $(b_1,a,a), (b_2,c,c), (d,a,c)$; it can be verified in polynomial time that $\varphi$ is a polymorphism of $\relstr A$.


\section{Conclusion}

Our main result, that a subuniverse of a finitely related finite algebra is absorbing if and only if it is J\'onsson absorbing, can be thought of as a ``decomposed version'' of the main result of~\cite{Bar13} which says that all the singletons of a finitely related finite algebra are absorbing if and only if they are J\'onsson absorbing.
In~\cite{Bvaleriote} it was shown that a finitely related finite algebra has Gumm terms if and only if it has cube terms. Is there a decomposed version of this result, that is, a useful notion of ``cube absorption'' which coincides with Gumm absorption for finitely related algebras?


\bibliographystyle{plain}
\normalsize
\baselineskip=17pt

\bibliography{absrelstr}

\def\cprime{$'$}
\begin{thebibliography}{10}

\bibitem{B11}
Libor Barto.
\newblock The dichotomy for conservative constraint satisfaction problems
  revisited.
\newblock In {\em 26th {A}nnual {IEEE} {S}ymposium on {L}ogic in {C}omputer
  {S}cience---{LICS} 2011}, pages 301--310. IEEE Computer Soc., Los Alamitos,
  CA, 2011.

\bibitem{Bar13}
Libor Barto.
\newblock Finitely related algebras in congruence distributive varieties have
  near unanimity terms.
\newblock {\em Canad. J. Math.}, 65(1):3--21, 2013.

\bibitem{Bvaleriote}
Libor Barto.
\newblock Finitely related algebras in congruence modular varieties have few
  subpowers.
\newblock Submitted, 2015.

\bibitem{BKazda}
Libor Barto and Alexandr Kazda.
\newblock Deciding absorption.
\newblock Submitted, arXiv:1512.07009, 2015.

\bibitem{BK09a}
Libor Barto and Marcin Kozik.
\newblock Congruence distributivity implies bounded width.
\newblock {\em SIAM Journal on Computing}, 39(4):1531--1542, 2009.

\bibitem{BK12}
Libor Barto and Marcin Kozik.
\newblock Absorbing subalgebras, cyclic terms, and the constraint satisfaction
  problem.
\newblock {\em Logical Methods in Computer Science}, 8(1), 2012.

\bibitem{BK14}
Libor Barto and Marcin Kozik.
\newblock Constraint satisfaction problems solvable by local consistency
  methods.
\newblock {\em J. ACM}, 61(1):3:1--3:19, January 2014.

\bibitem{BKN08b}
Libor Barto, Marcin Kozik, and Todd Niven.
\newblock The {CSP} dichotomy holds for digraphs with no sources and no sinks
  (a positive answer to a conjecture of {B}ang-{J}ensen and {H}ell).
\newblock {\em SIAM J. Comput.}, 38(5):1782--1802, 2008/09.

\bibitem{BKS15}
Libor Barto, Marcin Kozik, and David Stanovský.
\newblock Mal’tsev conditions, lack of absorption, and solvability.
\newblock {\em Algebra universalis}, 74(1-2):185--206, 2015.

\bibitem{BKW12}
Libor Barto, Marcin Kozik, and Ross Willard.
\newblock Near unanimity constraints have bounded pathwidth duality.
\newblock In {\em Proceedings of the 2012 27th {A}nnual {ACM}/{IEEE}
  {S}ymposium on {L}ogic in {C}omputer {S}cience}, pages 125--134. IEEE
  Computer Soc., Los Alamitos, CA, 2012.

\bibitem{Berg11}
C.~Bergman.
\newblock {\em Universal Algebra: Fundamentals and Selected Topics}.
\newblock Pure and Applied Mathematics. Taylor and Francis, 2011.

\bibitem{BKKR69}
V.~G. Bodnar{\v{c}}uk, L.~A. Kalu{\v{z}}nin, V.~N. Kotov, and B.~A. Romov.
\newblock Galois theory for {P}ost algebras. {I}, {II}.
\newblock {\em Kibernetika (Kiev)}, (3):1--10; ibid. 1969, no. 5, 1--9, 1969.

\bibitem{Bulin14}
Jakub Bul{\'i}n.
\newblock Decidability of absorption in relational structures of bounded width.
\newblock {\em Algebra universalis}, 72(1):15--28, 2014.

\bibitem{BS81}
Stanley~N. Burris and H.~P. Sankappanavar.
\newblock {\em A course in universal algebra}, volume~78 of {\em Graduate Texts
  in Mathematics}.
\newblock Springer-Verlag, New York, 1981.

\bibitem{G68}
David Geiger.
\newblock Closed systems of functions and predicates.
\newblock {\em Pacific J. Math.}, 27:95--100, 1968.

\bibitem{KKMM}
Alexandr Kazda, Marcin Kozik, Ralph McKenzie, and Matthew Moore.
\newblock Absorption and directed j\'onsson terms.
\newblock Submitted, 2015.

\bibitem{Willard10}
Ross Willard.
\newblock Testing expressibility is hard.
\newblock In David Cohen, editor, {\em Principles and Practice of Constraint
  Programming – CP 2010}, volume 6308 of {\em Lecture Notes in Computer
  Science}, pages 9--23. Springer Berlin Heidelberg, 2010.

\bibitem{Zhuk14}
Dmitriy~N. Zhuk.
\newblock The existence of a near-unanimity function is decidable.
\newblock {\em Algebra universalis}, 71(1):31--54, 2014.

\end{thebibliography}


\end{document}